\documentclass[12pt]{amsart}
\usepackage[margin = 0.9in]{geometry}

\begin{document}

\title[Fast Convergence for Point Estimates]{Fast Convergence for Weighted Least Squares Estimates}

\author{Andrey Sarantsev}

\begin{abstract}
It is well-known that maximum likelihood estimates converge faster than the classic square root rate if the Fisher information is infinite. This is often the case when the effective region depends on the estimated parameters, or when density has a singularity inside the effective region at a point dependent on the estimated parameters. We present a one-parameter family of bivariate absolutely continuous distributions on the half-space with smooth densities. The effective domain is always the same half-space and does not depend on this parameter. The order of magnitude for the weighted least squares estimate is asymptotically smaller than the classic square root rate. For the Gaussian variance mixture case, the maximum likelihood estimate coincides with this weighted least squares estimate.
\end{abstract}

\thispagestyle{empty}



\maketitle

\newtheorem{theorem}{Theorem}
\newtheorem{lemma}{Lemma}
\theoremstyle{definition}
\newtheorem{assumption}{Assumption}
\newtheorem{definition}{Definition}
\newtheorem{example}{Example}

\thispagestyle{empty}

\section{Introduction}

\subsection{Classic results} Take a (univariate or multivariate) family of absolutely continuous distributions, parameterized by $\theta \in \mathbb R$. Their Lebesgue density is $f(x, \theta)$. Take a sample of independent identically distributed (IID) random variables $X_1, \ldots, X_n$ from this distribution. Our goal is to study the maximum likelihood estimate (MLE) $\hat{\theta}_n$. Under classic results by Sir Ronald Fisher and others, this estimate is asymptotically normal:
$$
\sqrt{n}(\hat{\theta}_n - \theta) \stackrel{d}{\to} \mathcal N(\mathbf{0}, \Sigma^{-1}),
$$
where $\Sigma$ is the {\it Fisher information}, defined as the variance of the {\it score function} $\nabla_{\theta}\log f(x, \theta)$ with $x := X$ having density $f(x, \theta)$. See for example \cite[Theorem 10.1.12]{CasellaBerger}. Therefore, the classic convergence rate is $\sqrt{n}$: The distance between the MLE and the true value is inversely proportional to the square root of the number of data points. However, there exist cases when the convergence rate is faster: The classic example is the uniform distribution on $[0, \theta]$, where $\theta > 0$ is an unknown parameter. The MLE is then 
$$
\hat{\theta}_n = \max(X_1, \ldots, X_n),
$$
and converges to the true value with fast rate $n$ instead of $\sqrt{n}$, see \cite[Example 5.5.11]{CasellaBerger}:
$$
n(\theta - \hat{\theta}_n) \stackrel{d}{\to} \theta E,\quad n \to \infty,
$$
where $E$ is the standard exponential random variable. The reason for this fast convergence is dependence of the effective domain $[0, \theta]$ (and not just the likelihood function) upon the unknown parameter $\theta$. One can construct other such examples, with rates intermediate between $\sqrt{n}$ and $n$. See, for example, leave-one-out likelihood in \cite{LOO}, with super-effective rates $n^{\varepsilon + 1/2}$. Another case when the Fisher information is infinite is considered in \cite{Litvinova}. 

\subsection{Our contribution} A natural question is to find a family of continuous densities {\it with the same effective domain}, for which the MLE exists and is unique, but converges to the true value faster than $\sqrt{n}$. In this short note, we present a simple example of a bivariate family of densities with this property: Theorem~\ref{thm:main}. We can get any rate faster than $n^{0.5+\varepsilon}$ for example, $n$ or $n^2$. We can also get rates of the type 
$$
\sqrt{n\ln^{1 + \varepsilon}n},\quad \sqrt{n\ln\ln n}, \ldots
$$

\begin{assumption}
Take two random variables $Z$ and $W > 0$ with continuous  densities 
$$
g_W : (0, \infty) \to (0, \infty), \quad g_Z : \mathbb R \to (0, \infty).
$$
Also, $\mathbb E[Z] = 0$ and $\mathbb E[Z^2] = 1$. 
\label{asmp:basic}
\end{assumption}

An important case is the standard normal $Z$. 

\begin{definition}
The {\it variance mixture} is defined as 
\begin{equation}
\label{eq:repr}
X = \mu + \sqrt{W}Z
\end{equation}
If $Z \sim \mathcal N(0, 1)$, this is the {\it Gaussian variance mixture}. 
\end{definition}

These Gaussian mixtures (and their generalization, the {\it Gaussian mean-variance mixture}) are a classic study object. These include, for example, a large class of generalized hyperbolic distributions, \cite{BN}. For a multivariate generalization, see \cite{Hintz}. For semiparametric estimation, see \cite{Belo}. Then $X$ has a continuous Lebesgue density $g_X : (0, \infty) \to (0, \infty)$. The conditional distribution of $X$ given $W$ is $X\mid W \sim \mathcal N(\mu, W)$. The bivariate random variable $(X, W)$ has joint density
\begin{equation}
\label{eq:joint}
g(x, w; \mu) = \frac{g_W(w)}{\sqrt{w}}\cdot g_Z\left(\frac{x - \mu}{\sqrt{w}}\right).
\end{equation}

\begin{assumption} Take $(X_1, W_1), \ldots, (X_n, W_n)$, IID copies of the $(X, W)$ from~(\ref{eq:repr}).
\label{asmp:iid}
\end{assumption}

We have: $X_i\mid W_i \sim \mathcal N(\mu, W_i)$, and we can represent
$X_i = \mu + \sqrt{W_i}Z_i$ for IID copies $Z_1, \ldots, Z_n$ of $Z$ from Assumption~\ref{asmp:basic}. Applications include: 

\begin{itemize}
\item Measuring a quantity with observable variance (observation quality) $W_i$ dependent on the observation;
\item Observing stock index returns $X_i$ and volatility $W_i$ (measured by VIX, for example), and estimating long-term mean returns;
\item Connections with empirical Bayes estimation of multivariate normal. 
\end{itemize}

We are interested in the case when $\mathbb E[W^{-1}] = \infty$. Then, as we see in Lemma~\ref{lemma:Fisher}, the Fisher information is infinite. 

\subsection{Motivation} This research was inspired by our recent manuscript on $(X, W)$, where $W$ is Gamma and $X$ is Gaussian mean-variance mixture: $X\mid W \sim \mathcal N(\mu + \delta W, \sigma^2W)$. To our surprise, we established faster-than-usual convergence rate $\hat{\mu}_n \to \mu$ for some cases. In particular, if $W$ is exponential (and then $X$ is asymmetric Laplace), then the rate is $\sqrt{n\ln n}$. This works for the case of only variance mixture, without mixing using the mean, with $\delta = 0$ and symmetric Laplace $X$. There, the joint density of $(X, W)$ is continuous on $[0, \infty)\times \mathbb R$. 

In this short note, we generalize this idea for various densities of $W$ other than exponential/Gamma and for various densities of $Z$ other than standard Gaussian. We have only one parameter instead of five: De facto, we assume $\delta = 0$ and $\sigma = 1$. The marginal distribution of $W$ is assumed to be completely known and does not contain any parameters. 

\subsection{Overview of our results} Theorem~\ref{thm:main} for the Gaussian case and Theorem~\ref{thm:general} for the general case are our main results. In the Gaussian case, we compute the MLE. Although we know the exact distribution of this estimate in the Gaussian case, see Lemma~\ref{lemma:tech}, and therefore we do not need asymptotic results for statistical analysis, we present said results to give a desired example. Also, this proof is useful for the general case as well (when $Z$ is not Gaussian). 

In Theorem~\ref{thm:general}, we get weaker results for general $W$: The order of magnitude is the same. But we only get tightness instead of convergence. Therefore, we do not have any limiting distribution known. Also, we use weighted least squares (WLS) estimate, which can be explicitly computed. In the Gaussian case, this coincides with the MLE. But in the general case, these do not coincide; and it is not possible to compute the MLE explicitly. These results apply to any distribution $W$ (with mean zero and variance one), not just Gaussian. 

\subsection{Organization of the article} In Section 2, we present some further discussion and motivation for our research. We also discuss further possible directions of research. Section 3 discusses the Gaussian case, when $Z \sim \mathcal N(0, 1)$, and contains the first main result: Theorem~\ref{thm:main}. Section 4 is devoted to the general case, when $Z$ is the general random variable with mean zero and variance one, and contains the second main result: Theorem~\ref{thm:general}. Section 5 contains examples of possible super-efficient convergence rates faster than $\sqrt{N}$. It turns out we can get very diverse rates by changing the distribution of $W$. Section 6 is devoted to the complicated proof of Theorem~\ref{thm:main}.

\section{Discussion and further research} 

\subsection{Estimate choice: MLE vs WLS} Usually, the literature on super-efficiency (some of which is cited above) focuses on the MLE. Here, we made an unusual but deliberate choice to use the WLS. It is easy to compute WLS for both Gaussian and non-Gaussian $Z$. However, in the Gaussian case, this WLS, in fact, coincides with the MLE.

Working with the MLE in the non-Gaussian case would likely require us to impose additional assumptions: First, to get existence and uniqueness of said estimate, we would possibly need convexity in some form. Second, we will need to modify the proofs of asymptotic normality for the MLE in the classic case, when the Fisher information is finite, and convergence rate is the standard $\sqrt{n}$. This, in turn, will require interchange of differentiation and integration. We need  yet another assumption for this. 

In contrast, for WLS we do not need any of these assumptions. And the proof is straightforward. However, a valid criticism point is that we could get only tightness results instead of convergence (for WLS in the non-Gaussian case).

\subsection{Known variance interpretation} In practice, usually we observe data $X_1, \ldots, X_n$ which are independent identically distributed and therefore have the same (but unknown) mean and the same (but unknown) variance. Then the standard estimate $\overline{x}$ for the mean has standard convergence rate $\sqrt{n}$ (of course, if $X$ has finite second moment). 

But what if we know the variance $W_i$ of each observation $X_i$? What if this variance could be different for each observation? If the variance is separated from zero, the order of magnitude of these estimates will not change. Assume now that many observations $X_i$ have very small variance $W_i$. More rigorously, assume strong concentration of measure for $W$ around zero; or, equivalently, very heavy tails for $1/W$. Then estimates of $\mu$ are more precise and efficient. Changing the tail of $1/W$, we can adjust the convergence rate for this WLS estimate. For the rate, we can get any regularly varying function with index greater than $1/2$, and some functions with idex equal to $1/2$. 

\subsection{Further research} We stress that we could have generalized this for Gaussian mean-variance mixtures. This would involve three parameters $\mu, \delta, \sigma$ instead of only $\mu$. Moreover, we could have included other unknown parameters in the distribution of $W$, such as the shape and the scale for the Gamma distribution. However, we deliberately chose to present the simplest possible example in this short note. Generalizations are left for further research. 

One could connect this to empirical Bayes method. Also, we could consider the case when the mean $\mu$ differs from observation to observation: $X_i\mid W_i \sim \mathcal N(\mu_i, W_i)$, for $i = 1, \ldots, n$, but the vector $(\mu_1, \ldots, \mu_n)$ of such means belongs to a certain linear subspace. 

Finally, we could search for a univariate family of densities with similar properties as requested in this article. However, we feel skeptical: In our opinion, to get super-efficient estimates, we need to have a cusp or a discontinuity inside the effective domain, and its location must be dependent on the parameter. Alternatively, the effective domain itself must be dependent on the parameter. In our example, we do have a singularity dependent upon the parameter $\mu$. However, this point of singularity is $x = \mu$ and $w = 0$, which is {\it at the boundary} of the effective domain $\{(x, w)\mid x \in \mathbb R,\, w > 0\}$, and not inside. 

\section{The Gaussian Case}

In this section, we consider the case when Assumptions~\ref{asmp:basic} and~\ref{asmp:iid} hold, and $Z \sim \mathcal N(0, 1)$: Then the density of $Z$ is standard normal, 
\begin{equation}
\label{eq:Gaussian-standard}
g_Z(z) = \frac1{\sqrt{2\pi}}\exp\left[-z^2/2\right].
\end{equation}
The main result is Theorem~\ref{thm:main}, when (under the condition $\mathbb E[W^{-1}] = \infty$) we find an asymptotic result with an explicit rate of weak convergence and an explicit limit. This provides an example of a two-dimensional one-parameter family of smooth densities with the same effective domain (half-space), when the rate of convergence is faster than standard (super-efficient). 

\subsection{Preliminary remarks} This density~(\ref{eq:joint}) is a continuous function $g : \mathbb R\times (0, \infty) \to (0, \infty)$. The effective region is always the same: $\mathbb R\times(0, \infty)$, irrespective of the values of the parameter $\mu$. Inside this effective region, there are no singularities. 

For $x \ne \mu$, as $w \to 0$, then $g(x, w) \to 0$. But for $x = \mu$, as $w \to 0$, we get $g(\mu, w) \sim (2\pi w)^{-1/2}g_W(w)$. Only at $(\mu, 0)$, we might have some singularity, depending on the behavior of the density $g_W$ at zero. The marginal density of $X$ does not have any singularities. 

Using~(\ref{eq:Gaussian-standard}), we can rewrite the density~(\ref{eq:joint}) as a {\it curved exponential family}
$$
\frac{g_W(w)}{\sqrt{2\pi w}}\exp\left[-\frac{x^2}{2w} + \frac{x}{w}\mu - \frac{1}{2w}\mu^2\right]
$$
with {\it sufficient statistic} 
$$
\mathbf{T} = (\mathbf{T}_1, \mathbf{T}_2) := (x/w, -1/(2w)).
$$
The family is {\it curved} since we have a two-dimensional sufficient statistic but only one-dimensional parameter. 

\subsection{Fisher information} Let us compute the score function and the Fisher information for the joint density~(\ref{eq:joint}). The score function is
$$
s(x, w, \mu) = \frac{\partial \ln g(x, w, \mu)}{\partial \mu} = -\frac{x-\mu}{w}.
$$
The Fisher information is 
\begin{equation}
\label{eq:F}
\Sigma = \mathbb E[s^2(X, W, \mu)] = \mathbb E\left[\frac{(X - \mu)^2}{W^2}\right].
\end{equation}
Plugging~(\ref{eq:repr}) in~(\ref{eq:F}), we get: $\Sigma = \mathbb E[1/W]$. If this is finite, we have classic convergence rate $\sqrt{N}$. But we are interested in the case when $\Sigma = \infty$. This observation matches Lemma~\ref{lemma:Fisher} for the general case. 

\subsection{Computation of the estimate} We have one unknown parameter $\mu$. Now, pick a sample from this distribution: a sequence of IID random variables $(X_1, W_1), \ldots, (X_n, W_n)$. Write the {\it log-likelihood function}
\begin{equation}
\label{eq:LL}
\ell(\mu) := \sum\limits_{i=1}^n\ln g(X_i, W_i, \mu).
\end{equation} 
The {\it maximum likelihood estimate} (MLE) is the value $\hat{\mu}_N$ which maximizes the log-likelihood function $\ell$ from~(\ref{eq:LL}). If we had only one random variable $X$ observed, then it would not be so easy to solve for MLE. But knowledge of $W$ gives us critical additional information, so the solution is straightforward and explicit, below in~(\ref{eq:MLE}). Such MLE exists and is unique, see \cite[Section 1, formula (1)]{Han} or \cite[Lemma 1, Remark 1]{Weinstein}:
\begin{equation}
\label{eq:MLE}
\hat{\mu}_n = \frac{X_1/W_1 + \ldots + X_n/W_n}{1/W_1 + \ldots + 1/W_n}.
\end{equation}

\begin{lemma} Under Assumptions~\ref{asmp:basic} and~\ref{asmp:iid}, if $Z \sim \mathcal N(0, 1)$, the conditional distribution of $\hat{\mu}_n$ is 
\begin{equation}
\label{eq:conditional}
\hat{\mu}_n \mid W_1, \ldots, W_n \sim \mathcal N(\mu, S(n)^{-1}),
\end{equation}
with $S(n)$ given~(\ref{eq:final-repr}). 
\label{lemma:tech}
\end{lemma}

\begin{proof}
Apply~(\ref{eq:repr}) to each $X_i$:  We can represent each observation $X_i$ as 
\begin{equation}
\label{eq:X-W-i}
X_i = \mu + \sqrt{W_i}Z_i,\quad i = 1, \ldots, n.
\end{equation}
Plugging~(\ref{eq:X-W-i}) into the first equation in~(\ref{eq:MLE}), we get: 
\begin{align}
\label{eq:big}
\begin{split}
\hat{\mu}_n &= \frac{\mu/W_1 + \ldots + \mu/W_n}{1/W_1 + \ldots + 1/W_n}  \\ & + 
\frac{Z_1/\sqrt{W_1} + \ldots + Z_n/\sqrt{W_n}}{1/W_1 + \ldots + 1/W_n}.
\end{split}
\end{align}
The first term in the right-hand side of~(\ref{eq:big}) is equal to $\mu$, after everything else in the numerator and the denominator cancels out. The second term, conditional on the values of $W_1, \ldots, W_n$, is normally distributed with mean zero and variance 
\begin{align*}
&\frac{(1/\sqrt{W_1})^2 + \ldots + (1/\sqrt{W_n})^2}{(1/W_1 + \ldots + 1/W_n)^2} \\ & = \frac1{1/W_1 + \ldots + 1/W_n} = \frac1{S(n)}.
\end{align*}
This completes the proof of~(\ref{eq:conditional}). Thus we get~(\ref{eq:final-repr}).
\end{proof}

This obviates the need for asymptotic results as $n \to \infty$. We can compute confidence intervals or test hypotheses without these asymptotic results. But it is still interesting to consider them, because these give us exact asymptotics faster than 
$\sqrt{n}$. 

\subsection{The main result} Define the cumulative distribution function (CDF) of the random variable $W$ (equivalently, the distribution $Q$):
$$
G(u) = \int_0^ug_W(w)\,\mathrm{d}w = \mathbb P(W \le u),\quad u \in [0, \infty).
$$
Then the tail function of $1/W$ is $\mathbb P(1/W \ge u) = G(1/u)$ for $u > 0$. Using a well-known expression for the mean of a positive random variable, we get: 
\begin{equation}
\label{eq:expectation}
\mathbb E[1/W] = \int_0^{\infty}G(1/u)\,\mathrm{d}u.
\end{equation}
The function $G$ is strictly increasing on $[0, \infty)$, and its range is $[0, 1)$. Therefore, we can define its inverse $H : [0, 1) \to [0, \infty)$ as a strictly increasing function such that $H(G(u)) \equiv u$. Also, define $B(t) = 1/H(1/t)$, and let 
\begin{equation}
\label{eq:A-def}
A(t) = t\int_0^{B(t)}G(1/u)\,\mathrm{d}u,\quad t > 0,
\end{equation}

\begin{assumption}
The function $G$ is regularly varying at zero with index $\beta > 0$.  
\label{asmp:RV}
\end{assumption}

\begin{assumption}
We have $\mathbb E[1/W] = \infty$.
\label{asmp:infty}
\end{assumption}

In particular, if the density $g$ is regularly varying at zero with index $\beta - 1$, then the CDF $G$ is regularly varying at zero with index $\beta$, \cite[Appendix 6, Theorem 2.1(iv)]{Borovkov}. 

\begin{lemma} Under Assumption~\ref{asmp:RV}: 

\textsc{(A)} If $\beta > 1$, then Assumption~\ref{asmp:infty} does not hold. If $\beta < 1$, then Assumption~\ref{asmp:infty} holds. For the case $\beta = 1$, we cannot conclude either way. 

\textsc{(B)} The function $B$ is regularly varying with index $1/\beta$ at infinity. If $\beta = 1$, then the function $A$ is regularly varying with index 1 at infinity.

\textsc{(C)} If $\beta = 1$, then $A(t)/B(t) \to \infty$ as $t \to \infty$.
\label{lemma:beta}
\end{lemma}

\begin{proof} \textsc{(A)} We use~\eqref{eq:expectation}. By the properties of regularly varying functions, see \cite[Appendix 6, Theorem 2.1(iv)]{Borovkov} $u \mapsto G(1/u)$ is regularly varying at infinity with index $-\beta$. Therefore, the function $u \mapsto \int_0^uG(1/v)\,\mathrm{d}v$ is regularly varying at infinity with index $1 - \beta$, if only $\beta \ne 1$. Such function converges to infinity at infinity if $1 - \beta > 0$, and is bounded at infinity if $1 - \beta < 0$. The latter statement can be derived easily from \cite[Appendix 6, Theorem 2.1(ii)]{Borovkov}. This completes the proof of Lemma~\ref{lemma:beta}. 

Finally, consider the case $\beta = 1$. If $G(u) = u$ for $0 \le u \le 1$ (with uniform $W$ on $[0, 1]$), this gives us $\mathbb E[1/W] = \infty$. If $G(u) = u\ln^{\gamma}u$ for $0 \le u \le c$ (if $c > 0$ is small enough and $\gamma > 1$) this gives us $\mathbb E[1/W] < \infty$. 

\textsc{(B)} We apply \cite[Appendix 6, Theorem 2.1(v)]{Borovkov} to get that the inverse function $H(1/t)$ is regularly varying with index $-1/\beta$. Therefore, $B(t) = 1/H(1/t)$ is regularly varying with index $1/\beta$. The statement for the function $A$ follows from successive applications of classic results on regularly varying functions from the same monograph: \cite[Appendix 6, Theorem 2.1(iv), (v)]{Borovkov}. The function $G(1/u)$ is regularly varying at infinity with index $-1$. Therefore, the function 
$$
K: t \mapsto \int_0^tG(1/u)\,\mathrm{d}u
$$
is slowly varying at infinity. But the function $B$ is regularly varying at infinity with index $1$. Thus, the function $t \mapsto K(B(t))$ is also slowly varying at infinity. Finally, the function $A(t) = tK(B(t))$ is regularly varying at infinity with index 1. 

\textsc{(C)} Apply \cite[Appendix 6, Theorem 2.1(iv), (2.5)]{Borovkov} to the function $V(s) = G(1/s)$ (regularly varying at infinity with index $-1$), we get: 
\begin{equation}
\label{eq:L-def}
L(t) := \frac1{tV(t)}\int_0^tV(s)\,\mathrm{d}s \to \infty,\quad t \to \infty,
\end{equation}
Rewrite~(\ref{eq:L-def}) as follows:
\begin{equation}
\label{eq:L-int}
\int_0^tG(1/s)\,\mathrm{d}s = tV(t)L(V(t)) = tG(1/t)L(t).
\end{equation}
Plugging~(\ref{eq:L-int}) into~(\ref{eq:A-def}), and using 
$$
G(1/B(t)) = G(H(1/t)) = 1/t,
$$
we get the asymptotics:
\begin{align*}
\frac{A(t)}{B(t)} &= \frac{t}{B(t)}\int_0^{B(t)}G(1/s)\,\mathrm{d}s \\ & = \frac{t}{B(t)}B(t)G(1/B(t))L(B(t)) = L(B(t)) \to \infty.
\end{align*}
\end{proof}

For $\beta \in (0, 1)$, define $U_{\beta} > 0$ to be the stable subordinator with index $\beta$,  (see \cite[Chapter V, Theorem 3.5]{Steutel} or any other standard reference):  
$$
\mathbb E[e^{-\lambda U_{\beta}}] = \exp(-\Gamma(1 - \beta)\lambda^{\beta}),\quad \lambda > 0.
$$
We say that $f(t) \sim g(t)$ as $t \to a$ for $a = 0, \infty$, if $f(t)/g(t) \to 1$ as $t \to a$. 

\begin{definition} A function $f : (0, \infty) \to (0, \infty)$ is called {\it regularly  varying at} $a$ with {\it index} $b$, where $a = 0$ or $a = \infty$, if for any $s > 0$ we have: 
$$
\lim\limits_{t \to a}\frac{f(st)}{f(t)} = s^b\quad \mbox{for all}\quad s > 0.
$$
For the case $b = 0$, we say $f$ is {\it slowly varying at} $a$. 
\end{definition}

Below is the main result of this short note. 

\begin{theorem} Impose Assumptions~\ref{asmp:basic},~\ref{asmp:iid},~\ref{asmp:RV}. Let $Z \sim \mathcal N(0, 1)$. Then the MLE $\hat{\mu}_n$ satisfies:

\textsc{(A)} If $\beta \in (0, 1)$, take a random variable $V \sim \mathcal N(0, 1)$ independent of $U_{\beta}$. Then
\begin{equation}
\label{eq:beta<1}
\sqrt{B(n)}\cdot(\hat{\mu}_n - \mu) \stackrel{d}{\to} U_{\beta}V,\quad n \to \infty,
\end{equation}

\textsc{(B)} If $\beta = 1$, then under Assumption~\ref{asmp:infty},
\begin{equation}
\label{eq:beta=1}
\sqrt{A(n)}\cdot(\hat{\mu}_n - \mu) \stackrel{d}{\to} V,\quad n \to \infty.
\end{equation}
\label{thm:main}
\end{theorem}

The proof is postponed until the Appendix, because it is complicated. 

\section{The General Case}

If $Z$ is not necessarily normal, we do not have any exponential family of distributions any more. What is more, it is hard to find the maximum likelihood estimate explicitly. Instead, we shall recycle the estimate~\eqref{eq:MLE}, although it does not necessarily coincide with the maximum likelihood estimate any more. The reason why we use it is because this is the weighted least squares (WLS) estimate: The one which minimizes the objective function
\begin{equation}
\label{eq:objective}
L(\mu) = \sum\limits_{i=1}^n\frac{(X_i - \mu)^2}{W_i}. 
\end{equation}
It is a standard exercise to check that the $\hat{\mu}$ from~(\ref{eq:MLE}) minimizes~(\ref{eq:objective}).  We do not know the distribution of~(\ref{eq:MLE}) any more in the exact form, as in Lemma~\ref{lemma:tech} for the Gaussian case. But we do have an asymptotic result for this estimate, albeit weaker than for the Gaussian case. 

\begin{theorem} Impose Assumptions~\ref{asmp:basic},~\ref{asmp:iid},~\ref{asmp:RV}. Then the WLS estimate $\hat{\mu}_n$ from~\eqref{eq:MLE} satisfies:

\textsc{(A)} If $\beta \in (0, 1)$, then the following sequence is tight:
$$
\{\sqrt{B(n)}\cdot(\hat{\mu}_n - \mu),\quad n = 1, 2, \ldots\}
$$

\textsc{(B)} If $\beta = 1$, then the following sequence is tight: 
$$
\{\sqrt{A(n)}\cdot(\hat{\mu}_n - \mu),\quad n = 1, 2, \ldots\}
$$
\label{thm:general}
\end{theorem}

\begin{proof} The computation similar to the one in Lemma~\ref{lemma:tech} shows that conditional mean and variance of 
$$
\hat{\mu}_n\mid W_1, \ldots, W_n
$$
are $\mu$ and $S(n)^{-1}$, using the notation for the $S(n)$ from~(\ref{eq:final-repr}). Therefore, 
$$
\mathbb E\left[(\hat{\mu}_n - \mu)^2\mid W_1, W_2,\ldots\right] = \frac1{S(n)}.
$$
Since $S(n)$ depends only on $W_1, W_2, \ldots$, we can rewrite 
$$
\mathrm{Var}(\sqrt{S(n)}\hat{\mu}_n\mid W_1, W_2, \ldots) = 1.
$$
By Chebyshev's inequality, for any $\varepsilon > 0$ we get:
\begin{equation}
\label{eq:est-1}
\mathbb P\left(\sqrt{S(n)}|\hat{\mu}_n - \mu| \ge \varepsilon\mid W_1, W_2, \ldots\right) \le \frac{1}{\varepsilon^2}.
\end{equation}
Now take any sequence $C(n)$ of positive numbers such that, as $n \to \infty$, 
\begin{equation}
\label{eq:choice}
\frac{C(n)}{B(n)} \to 0,\quad \beta \in (0, 1);\quad \frac{C(n)}{A(n)} \to 0,\quad \beta = 1.
\end{equation}
 Using the results of~\eqref{eq:main-A} and~\eqref{eq:main-B}, we get 
$C(n)/S(n) \stackrel{d}{\to} 0$ as $n \to \infty$. Therefore, for any $\delta > 0$ we have:
\begin{equation}
\label{eq:est-2}
\mathbb P\left(C(n)/S(n) > \delta\right) \to 0,\quad n \to \infty.
\end{equation}
Our goal is to study the probability
$$
\mathbb P\left(\sqrt{C(n)}|\hat{\mu}_n - \mu| \ge \varepsilon\mid W_1, W_2, \ldots\right)
$$
We can split this probability in two terms. Use~(\ref{eq:est-1}) for $\delta^{-1/2}\varepsilon$ instead of $\varepsilon$:
\begin{align*}
&\mathbb P\left(\sqrt{C(n)}|\hat{\mu}_n - \mu| \ge \varepsilon\mid W_1, W_2, \ldots\right) \\ & \le \mathbb P\left(\sqrt{S(n)}|\hat{\mu}_n - \mu| \ge \delta^{-1/2}\varepsilon\mid W_1, W_2, \ldots\right) \\ & + \mathbb P\left(C(n)/S(n) > \delta\right) 
\\ & \le \frac{\delta}{\varepsilon^2} + \mathbb P\left(C(n)/S(n) > \delta\right).
\end{align*}
Using~(\ref{eq:est-2}), and letting $n \to \infty$, we get:  
$$
\varlimsup\limits_{n \to \infty}\mathbb P\left(\sqrt{C(n)}|\hat{\mu}_n - \mu| \ge \varepsilon\mid W_1, W_2, \ldots\right) \le \frac{\delta}{\varepsilon^2}.
$$
Using that $\delta > 0$ is arbitrary, we can conclude
\begin{equation}
\label{eq:final-est}
\mathbb P\left(\sqrt{C(n)}|\hat{\mu}_n - \mu| \ge \varepsilon\mid W_1, W_2, \ldots\right) \to 0.
\end{equation}
Take the expectation in~(\ref{eq:final-est}). In the left-hand side, by the total probability formula, there will be the unconditional probability. However, this conditional probability is always between zero and one. Using Lebesgue dominated convergence theorem, we conclude that the unconditional probability tends to zero:
$$
\mathbb P\left(\sqrt{C(n)}|\hat{\mu}_n - \mu| \ge \varepsilon\right) \to 0,\quad n \to \infty.
$$
Since this holds for any $C(n)$ from~(\ref{eq:choice}), we complete the proof of Theorem~\ref{thm:general}. 
\end{proof}

Unfortunately, we do not have exact distribution of the limiting law or laws. This is a significant shortcoming, but a tightness result is better than nothing. For completeness, we present the result on Fisher information below.

\begin{lemma} If $g_Z$ is differentiable, and $\mathbb E[1/W] = \infty$, then the Fisher information is infinite.
\label{lemma:Fisher}
\end{lemma}

\begin{proof}
Compute the score function
$$
s(x, w, \mu) = \frac{\partial \ln g(x, w; \mu)}{\partial \mu}.
$$
After taking logarithms in~(\ref{eq:joint}), we have a sum, and only the last term there depends on $\mu$:
$$
\ln g_Z\left(\frac{x - \mu}{\sqrt{w}}\right)
$$
Its derivative with respect to $\mu$ is written by the chain rule:
\begin{equation}
\label{eq:score-general}
s(x, w, \mu) = -\frac1{\sqrt{w}}(\ln g_Z)'\left(\frac{x - \mu}{\sqrt{w}}\right).
\end{equation}
Square this expression~(\ref{eq:score-general}), plug in $z = Z$ and $w = W$. Inside the argument of the function $(\ln g_Z)'$, we get: $(X - \mu)/\sqrt{W} = Z$, which is independent of $W$. Thus, 
\begin{equation}
\label{eq:score-square}
s^2(X, W, \mu) = -\frac1{\sqrt{W}}(\ln g_Z)'(Z).
\end{equation}
Take the expected value of~(\ref{eq:score-square}). Using independence of $W$ and $Z$, we get $\mathbb E[1/W]\cdot \mathbb E[(\ln g_Z)'(Z)^2] = \infty$.  
\end{proof}

Finally, we present the asymptotic normality result for the regular case, when $\mathbb E[1/W] < \infty$. It is quite elementary, but we feel the need to include this for clarity of exposition. 

\begin{lemma}
\label{lemma:reg}
Impose Assumptions~\ref{asmp:basic},~\ref{asmp:iid}, but suppose that Assumption~\ref{asmp:infty} does not hold. Then the WLS estimate $\hat{\mu}_n$ satisfies
$$
\sqrt{n}(\hat{\mu}_n - \mu) \stackrel{d}{\to} \mathcal N\left(0, (\mathbb E[1/W])^{-1}\right).
$$
\end{lemma}

\begin{proof}
As in the proof of Lemma~\ref{lemma:tech}, we can represent 
$$
\hat{\mu}_n - \mu = \frac1{S_n}\sum\limits_{i=1}^n\frac{Z_i}{\sqrt{W_i}}.
$$
By the Law of Large Numbers, 
\begin{equation}
\label{eq:LLN}
\frac{S_n}{n} \stackrel{d}{\to} \mathbb E[1/W].
\end{equation}
By independence of $Z_i$ and $W_i$, we get: 
\begin{align*}
\mathbb E(Z_i/\sqrt{W_i}) &= \mathbb E[Z_i]\cdot\mathbb E[1/\sqrt{W_i}] = 0; \\
\mathbb E(Z_i/\sqrt{W_i})^2 &= \mathbb E[Z_i^2]\cdot\mathbb E[W_i^{-1}] = \mathbb E[1/W].
\end{align*}
By the Central Limit Theorem, 
\begin{equation}
\label{eq:CLT}
\frac1{\sqrt{n}}\sum\limits_{i=1}^n\frac{Z_i}{\sqrt{W_i}} \stackrel{d}{\to} \mathcal N(0, \mathbb E[1/W]). 
\end{equation}
Combining~(\ref{eq:LLN}) and~(\ref{eq:CLT}), and using the Slutsky theorem, \cite[Theorem 5.5.17]{CasellaBerger}, we complete the proof of Lemma~\ref{lemma:reg}.
\end{proof}

\section{Convergence Rates} 

In this section, we play with possible convergence rates. This applies to both main results: Theorems~\ref{thm:main} and~\ref{thm:general}. 

\subsection{Two cases} Let us discuss special cases of Theorem~\ref{thm:main}. Various choices of $g_W$ and therefore $G$ lead to various convergence rates. First, consider the case $\beta < 1$. 

\begin{lemma} If $\beta \in (0, 1)$, take any regularly varying function $C(t) > 0$ at infinity with index $\alpha > 0.5$, and with $C' > 0$ everywhere. Then there exists a distribution $Q$ on $(0, \infty)$ such that $B(t) = C(t)$ in Theorem~\ref{thm:main}.
\end{lemma}

\begin{proof} Indeed, just take an inverse function $F$ to $C^2$: $F(C^2(t)) = t$. The function $C^2$ is strictly positive and has strictly positive derivative. Therefore, the inverse function also has strictly positive derivative. By the results of \cite[Appendix 6, Theorem 2.1(v)]{Borovkov}, this function is $1/(2\alpha)$ at infinity. Then define $G(u) = 1/F(1/u)$ to be the cumulative distribution function of $W$. Following \cite[Appendix 6, Theorem 2.1]{Borovkov}, this is a regularly varying function at zero with index $\beta = 1/(2\alpha) \in (1/2, 1)$. 
\end{proof}

It is more complicated to find all possible convergence rates for $\beta = 1$. We leave it for the future resesarch.  Some examples are below. 

\subsection{Examples} First, consider the case $\beta \in (0, 1)$.

\begin{example} Fix $\gamma \in \mathbb R$, and assume  
\begin{equation}
\label{eq:G-asymp}
G(u) \sim \mathrm{const}\cdot u^{\beta}|\ln u|^{\gamma},\quad u \to 0.
\end{equation}
Take the logarithm: $\ln G(u) \sim \beta\ln u$. Therefore, $\ln v \sim \beta\ln H(v)$. 
A standard exercise shows that the inverse function $H$ of $G$ satisfies
$$
H(u) \sim \mathrm{const} \cdot u^{1/\beta}|\ln u|^{-\gamma/\beta}
$$
as $u \to 0$. And the function $B(t) = 1/H(1/t)$ satisfies
$$
B(t) \sim \mathrm{const}\cdot t^{1/\beta}(\ln t)^{\gamma/\beta},\quad t \to \infty. 
$$
Therefore, the convergence rate for $\beta \in (1/2, 1)$ is 
$$
\sqrt{B(t)} \sim \mathrm{const}\cdot t^{1/(2\beta)}(\ln t)^{\gamma/(2\beta)}.
$$
By choosing $a = 1/(2\beta)$ and $b = \gamma/(2\beta)$, we can make any convergence rate $t^a(\ln t)^b$ for any $a \in (0.5, 1)$ and $b$.  
\end{example}

For $\beta = 1$, it is harder to compute the convergence rate. From~(\ref{eq:G-asymp}), we get: 
$$
G(1/u) \sim \mathrm{const}\cdot u^{-1}(\ln u)^{\gamma}.
$$
Assumption~\ref{asmp:infty} holds if and only if $\gamma \le -1$. Here, in turn, we consider two cases: $\gamma = -1$ and $\gamma < -1$. 

\begin{example} For $\gamma = -1$, we have $G(u) \sim \mathrm{const}\cdot u^{\beta}|\ln u|^{-1}$. We compute
$$
\int_0^{t}G(1/u)\,\mathrm{d}u \sim \mathrm{const}\cdot \ln\ln t,\quad t \to \infty.
$$
Since $B(t) \sim \mathrm{const}\cdot t(\ln t)^{-1}$, we get: 
$$
A(t) \sim \mathrm{const}\cdot t\ln\ln t
$$
as $t \to \infty$. The convergence rate is 
$$
\sqrt{A(t)} \sim \mathrm{const}\cdot \sqrt{t\ln\ln t}.
$$
\end{example}

\begin{example} Next, for $\gamma < -1$, we have $G(u) \sim \mathrm{const}\cdot u^{\beta}|\ln u|^{-\gamma}$. We get: As $t \to \infty$, 
$$
\int_0^{t}G(1/u)\,\mathrm{d}u \sim \mathrm{const}\cdot (\ln t)^{1+\gamma}.
$$
Also, $B(t) \sim \mathrm{const}\cdot t(\ln t)^{\gamma}$. Therefore, as $t \to \infty$, 
$$
A(t) = t\int_0^{B(t)}G(1/u)\,\mathrm{d}u = \mathrm{const}\cdot t(\ln t)^{1 + \gamma}.
$$
Therefore, the convergence rate is 
$$
\sqrt{A(t)} \sim \mathrm{const}\cdot t^{1/2}(\ln t)^{(1 + \gamma)/2}.
$$
By choosing $b = (1+\gamma)/2$, we can make any convergence rate of the type $t^{1/2}\cdot (\ln t)^b$ for any $b > 0$. 
\end{example}

Similarly, we can have convergence rates of the type $t^a\cdot (\ln t)^b\cdot (\ln\ln t)^c$, for $a \ge 1/2$, $b, c \in \mathbb R$, so that if $a = 1/2$, then either $b > 0$, or $b = 0$ and $c > 0$. We can have even more iterations of logarithmic functions, as long as the rate is asymptotically faster than the standard convergence rate $t^{1/2}$.

\section{Proof of Theorem~\ref{thm:main}}

\subsection{Overview of the proof} The proof is based on classic results about domains of attraction and stable limit theorems, taken from \cite[Chapter 8, Theorem 8.1]{Borovkov}, also stated and proved in \cite[Appendix 7, Theorem 1.1]{Borovkov}. First, in Lemma~\ref{lemma:tech} we show we can rewrite with a $Z \sim \mathcal N(0, 1)$ independent of $W_1, \ldots, W_n \sim Q$: 
\begin{equation}
\label{eq:final-repr}
\hat{\mu}_n - \mu = \frac{V}{\sqrt{S(n)}},\quad S(n) := \sum\limits_{i=1}^nW_i^{-1}.
\end{equation}
For part (A), where $\beta \in (0, 1)$, we show:
\begin{equation}
\label{eq:main-A}
\frac{S(n)}{B(n)} \stackrel{d}{\to} U_{\beta},\quad n \to \infty.
\end{equation}
For part (B), where $\beta = 1$, it sufffices to show that 
\begin{equation}
\label{eq:main-B}
\frac{S(n)}{A(n)} \stackrel{d}{\to} 1,\quad n \to \infty.
\end{equation}
Combining~(\ref{eq:main-A}) with~(\ref{eq:final-repr}), we get~(\ref{eq:beta<1}) and prove (A).  Combining~(\ref{eq:main-B}) with~(\ref{eq:final-repr}), we get~(\ref{eq:beta=1}) and prove (B). Thus, we complete the proof of Theorem~\ref{thm:main}. 

\subsection{Proof of~(\ref{eq:main-A})} To show~(\ref{eq:main-A}), we apply the aforementioned results from \cite{Borovkov}: Chapter 8, Section 8, Theorem 8.1(i). We need to match the notation: $\xi_k := W_k^{-1}$, 
$$
F_+(t) = \mathbb P(1/W \ge t) = \mathbb P(W \le 1/t) = G(1/t);
$$
and $F_-(t) = 0$ because $1/W > 0$; 
$$
F_0(t) = F_+(t) + F_-(t) = G(1/t);
$$
and this function is regularly varying at infinity with index $-\beta$. Moreover, $b(n) = F_0^{(-1)}(1/n)$, where $F_0^{(-1)}(t) = 1/H(t)$ is the inverse of $F_0(t) = G(1/t)$. Thus $b(n) := B(n)$. Next, $V = F_+$ and $W = 0$. Therefore, $\rho_+ = 1$ and $\rho = 1$. Also, we use discussion on one-sided stable distributions after \cite[Chapter 8, Theorem 8.1]{Borovkov}. This shows that $\zeta^{(1, \rho)}$ in the notation of \cite{Borovkov} is, in fact, the stable subordinator $U_{\beta}$. Finally, we need  Assumption~\ref{asmp:infty}. Otherwise, per discussion before \cite[Chapter 8, Theorem 8.1]{Borovkov}, we would need to subtract this mean to get such weak convergence. The rest of~(\ref{eq:beta<1}) simply follows from the aforementioned theorem \cite[Chapter 8, Theorem 8.1]{Borovkov}. 

\subsection{Proof of~(\ref{eq:main-B})} Similarly to the case (A), we apply the aforementioned results from \cite{Borovkov}: Chapter 8, Section 8, Theorem 8.1(ii). We have the same notation, and some new notation:
$$
V_I(t) = \int_0^tV(y)\,\mathrm{d}y = \int_0^tG(1/y)\,\mathrm{d}y,\quad W_I = 0.
$$
Define $C \approx 0.5772$ to be the Euler's constant. We get $\rho = 1$, since there is no left tail. In the notation of \cite[Chapter 8, Section 8]{Borovkov}, 
\begin{align}
\label{eq:appendix}
\begin{split}
A_n &= \frac{n}{b(n)}\left[V_I(b(n)) - W_I(b(n))\right] - \rho C \\ & = \frac{n}{B(n)}\int_0^{B(n)}G(1/y)\,\mathrm{d}y - C.
\end{split}
\end{align}
Applying Lemma~\ref{lemma:beta}~\textsc{(C)} to~(\ref{eq:appendix}), we get:
$$
A_n = \frac{A(n)}{B(n)} - C \sim \frac{A(n)}{B(n)},\quad n \to \infty.
$$
The main result from \cite[Chapter 8, Theorem 8.1, case $\beta = 1$]{Borovkov}: A weak convergence to a limit $\zeta$ (whose exact distribution is not important for us): 
\begin{equation}
\label{eq:fundamental}
\frac{S(n)}{B(n)} - A_n \stackrel{d}{\to} \zeta,\quad n \to \infty.
\end{equation}
It suffices to note that $B(n)A_n \sim A(n)$ as $n \to \infty$. Divide~(\ref{eq:fundamental}) by $A_n$, and apply Lemma~\ref{lemma:beta}~\textsc{(C)} again to finish the proof of~(\ref{eq:beta=1}):  
$$
\frac{S(n)}{B(n)A_n} \stackrel{d}{\to} 1,\quad n \to \infty.
$$

\section*{Acknowledgments}
The author thanks the Department of Mathematics and Statistics, University of Nevada, Reno, for welcoming atmosphere for research. The author also thanks the referee and the editor for useful comments which led to a significant improvement. 

The author received no external funding.

\end{document}